\documentclass[11pt,a4paper]{article}

\usepackage[a4paper,left=1.08in,right=1.08in,top=0.94in,bottom=0.76in]{geometry}
\usepackage{amsmath,amssymb,amsthm}
\usepackage{microtype}
\usepackage{xcolor}
\usepackage{hyperref}
\usepackage{cite}

\hypersetup{
  colorlinks=true,
  linkcolor=red,
  citecolor=cyan!60!blue,
  urlcolor=blue
}

\makeatletter
\def\leftharpoonfill@{\arrowfill@\leftharpoonup\relbar\relbar}
\def\rightharpoonfill@{\arrowfill@\relbar\relbar\rightharpoonup}
\newcommand\rbjt{\mathpalette{\overarrow@\rightharpoonfill@}}
\newcommand\lbjt{\mathpalette{\overarrow@\leftharpoonfill@}}
\makeatother

\numberwithin{equation}{section}

\newtheorem{theorem}{Theorem}
\newtheorem{lemma}{Lemma}
\newtheorem{corollary}{Corollary}
\newtheorem{conjecture}{Conjecture}
\newtheorem{proposition}{Proposition}

\newcommand{\pico}{\pi_{\mathrm{co}}}
\newcommand{\piunif}{\pi_{\therefore}}

\title{Any $k$-graph with zero $\ell$-degree Tur\'an density is layered}
\author{Jiabao YANG, Xiaona FANG, Yaojun CHEN\footnote{Corresponding author. Email: yaojunc@nju.edu.cn}\\
 \small{School of Mathematics, Nanjing University, Nanjing 210093, P.R. CHINA}}
\date{}

\begin{document}
\pagestyle{plain}
\maketitle
\vspace{-1.1em}

\begin{abstract}
The codegree Tur\'an density $\pi_{\mathrm{co}}(F)$ is the supremum over all $\gamma \in [0,1)$ such that, for arbitrarily large $n$, there exists an $n$-vertex $F$-free $k$-graph $H$ whose every $(k-1)$-subset of vertices lies in at least $\gamma n$ edges.
Ding, Lamaison, Liu, Wang, and Yang (JLMS, 2025) studied the problem of what 3-graphs $F$ satisfy $\pi_{\mathrm{co}}(F) = 0$. They introduced layered $3$-graphs and 
conjectured that a $3$-graph has zero codegree Tur\'an density if and only if it is layered and has zero uniform Tur\'an density. 
For $k\ge 3$,  a $k$-graph is called layered if its vertices can be labelled so that every edge has a unique maximum label and two edges with the same maximum label have the same label multiset.  
In this paper, we show that every non-layered $k$-graph $F$ on $m$ vertices satisfies
\[
  \pico(F)\ge q_{k,m}^{-q_{k,m}}>0,
  \quad \text{where}\quad 
  q_{k,m}=\frac{(k-1)^{m+1}-1}{k-2},
\]
which implies  any $k$-graph with zero $\ell$-degree Tur\'an density is layered, and the case $k=3$ confirms the conjecture of Ding, Lamaison, Liu, Wang, and Yang.   

\vskip 2mm
\noindent\textbf{Keywords.} Layered $k$-graph; codegree Tur\'an density; $\ell$-degree Tur\'an density  
\end{abstract}

\section{Introduction}

A $k$-graph is a $k$-uniform hypergraph.
The Tur\'{a}n number of an $k$-graph $F$, denoted by $\mathrm{ex}(n,F)$, is the maximum number of edges in an $n$-vertex $k$-graph  that contains no copy of $F$.
Tur\'an-type problems form a central topic in extremal combinatorics.  
To describe the limiting behavior of $\mathrm{ex}(n,F)$, we define the Tur\'an density of $F$ by
$$\pi(F)=\lim_{n\to\infty} \frac{\mathrm{ex}(n,F)}{\binom {n}{k}}.$$
The limit exists by the classical averaging theorem of Katona, Nemetz and Simonovits~\cite{KNS1964}.

A natural variant of Tur\'an density, introduced by Mubayi and Zhao~\cite{mubayi}, is the codegree Tur\'an density. For a $k$-graph $H$ and a vertex set $S\subseteq V(H)$, let $d_H(S)$ denote the number of edges containing $S$. 
The minimum codegree of $H$, denoted by $\delta_{\mathrm{co}}(H)$, is defined as the minimum of $d_H(S)$ taken over all $(k-1)$-subsets $S$ of $V(H)$. 
The codegree Tur\'an number $\mathrm{ex}_{\mathrm{co}}(n,F)$ is the largest possible value of $\delta_{\mathrm{co}}(H)$ among all $n$-vertex $F$-free $k$-graphs $H$, 
and the corresponding codegree Tur\'an density is
\[
\pi_{\mathrm{co}}(F):=\lim_{n\to\infty}\frac{\mathrm{ex}_{\mathrm{co}}(n,F)}{n}.
\]
It is known that this limit always exists~\cite{mubayi}.

The $\ell$-degree Tur\'an density is a natural generalization of both the classical Tur\'an density and the codegree Tur\'an density.
Let $H$ be a $k$-graph. For $S \subseteq V(H)$ with $|S| < k$, the minimum $\ell$-degree $\delta_\ell(H)$ is the minimum of $d_H(S)$ over all $\ell$-subsets $S$.
Given a $k$-graph $F$, the $\ell$-degree Tur\'an number $\mathrm{ex}_\ell(n,F)$ is the maximum $\delta_\ell(H)$ that an $n$-vertex $F$-free $k$-graph $H$ can admit. The corresponding $\ell$-degree Tur\'an density is
\[
\pi_\ell(F) = \lim_{n \to \infty} \frac{\mathrm{ex}_\ell(n,F)}{\binom{n}{k-\ell}}.
\]
This limit always exists \cite{LoMarkstrom2014}, and for every $k$-graph $F$,
\[\pico(F)=
\pi_{k-1}(F) \le \pi_{k-2}(F) \le \cdots \le \pi_2(F) \le \pi_1(F) = \pi(F).
\]

For $k\geq3$, determining the exact codegree Tur\'an density of a $k$-graph is difficult in general. 
Nagle~\cite{Nagle1999} conjectured in 1999 that $\pi_{\mathrm{co}}(K_4^{(3)-})=1/4$, and Czygrinow and Nagle~\cite{CzygrinowNagle2001} later conjectured that $\pi_{\mathrm{co}}(K_4^{(3)})=1/2$. 
Falgas-Ravry, Pikhurko, Vaughan and Volec~\cite{FPRVV2023} proved the first conjecture using flag algebras, while the exact value of $\pi_{\mathrm{co}}(K_4^{(3)})$ remains unknown.
Exact codegree Tur\'an densities are known for only a few graphs. Mubayi~\cite{Mubayi2005} proved that $\pi_{\mathrm{co}}(\mathbb F)=1/2$, where $\mathbb F$ is the Fano plane. And the codegree Tur\'an density for other projective geometries are further studied by several researchers \cite{keevash,zhang,fang}.
Falgas-Ravry, Marchant, Pikhurko and Vaughan~\cite{FRMPV2015} proved that $\pi_{\mathrm{co}}(F_{3,2})=1/3$, where $V(F_{3,2})=[5]$ and $E(F_{3,2})=\{123,124,125,345\}$. Piga, Sales and
Sch\"ulke~\cite{PigaSalesSchulke2023} proved that $\pi_{\mathrm{co}}(C_\ell^{(3)-})=0$ for every $\ell\geq5$, where $C_\ell^{(3)-}$ is the tight $3$-uniform cycle of length
$\ell$ with one edge removed.
 Further results on the $k$-graphs $C_\ell^{(k)}$ and $C_\ell^{(k)-}$ can be found in \cite{Ma1, Ma2, Piga, Sarkies}. For $\ell$-degree Tur\'an density, Ai, Ding, Liu and Yang \cite{AiDingLiuYang2026Transfer} introduced tree suspensions and transfer functions for the single-forbidden $\ell$-degree Tur\'an spectrum.

These results also show that codegree Tur\'an density may have quite different behaviors for different forbidden $3$-graphs.
A hypergraph is said to be linear if every pair of distinct hyperedges intersects in at most one vertex.
Ding, Lamaison, Liu, Wang, and Yang \cite{DLLWY} studied the problem of what 3-graphs $F$ satisfy $\pi_{\mathrm{co}}(F) = 0$. 
They proposed a conjecture and reduced the problem to the linear $3$-graph case. 
Before stating their conjectures, we introduce the following notations.

For $U\subseteq V(H)$, let $e(U)$ be the number of edges of $H$ contained in $U$. 
The uniform Tur\'an density $\piunif(F)$ is the supremum of all $d\in[0,1]$ such that, for every $\eta>0$ and $n_0\in\mathbb N$, there is an $F$-free $3$-graph $H$ on at least $n_0$ vertices for which
\(
e(U)\ge d\binom{|U|}{3}-\eta |V(H)|^3
\)
for every $U\subseteq V(H)$.

\vskip 0.5em

Given a $k$-graph $F$, let $f:V(F)\to\mathbb N$ be a function.
For $e\in E(F)$, let $f(e)$ denote the multiset $\{f(v):v\in e\}$.
The function $f$ is a \emph{layered function} of $F$ if it satisfies the following two conditions:
 \begin{itemize}
  \setlength{\itemsep}{1pt}
  \setlength{\parsep}{1pt}
  \setlength{\parskip}{1pt}
  \item[(A1)] Each edge has exactly one vertex whose label is the
  maximum within the edge.
  \item[(A2)] If
  $\max f(e)=\max f(e')$ for $e,e'\in E(F)$, then $f(e)=f(e')$.
\end{itemize}
A $k$-graph is \emph{layered} if it has a layered function. It is worth noting that a layered $k$-graph for $k=3$  was defined in \cite{DLLWY},
which requires the layered function to satisfy a third property: 
\begin{itemize}
   \item[(A3)] If $|f(e)\cap f(e')|\ge k-1$ for $e,e'\in E(F)$, then $f(e)=f(e')$.
\end{itemize}
As also pointed out in \cite{DLLWY}, if there is a function satisfying (A1) and (A2), then there is one satisfying (A1), (A2) and (A3), that is, it is not necessary to require the function satisfying (A3).
In fact, this holds for general $k$-graphs,  see Proposition \ref{layered} in Section \ref{sec4}.

\vskip 0.2em
Ding, Lamaison, Liu, Wang, and Yang~\cite{DLLWY} studied the problem of what 3-graphs $F$ satisfy $\pi_{\mathrm{co}}(F) = 0$, and established the following result. 

\begin{theorem}[Ding, Lamaison, Liu, Wang, and Yang~\cite{DLLWY}]\label{thm:known}
Let $F$ be a $3$-graph.
\begin{enumerate}
 \setlength{\itemsep}{1pt}
  \setlength{\parsep}{1pt}
  \setlength{\parskip}{1pt}
 
\item[(1)] If $\pico(F)=0$, then $\piunif(F)=0$.
\item[(2)] If $F$ is layered, then $\pico(F)=0$ if and only if $\piunif(F)=0$.
\end{enumerate}
\end{theorem}
Moreover, they \cite{DLLWY} proposed the following two equivalent conjectures.

\begin{conjecture}[Ding, Lamaison, Liu, Wang, and Yang~\cite{DLLWY}]\label{conj:1.6}
For a $3$-graph $F$, $\pi_{\mathrm{co}}(F) = 0$ if and only if $F$ is layered and satisfies~$\piunif(F)=0$.
\end{conjecture}

\begin{conjecture}[Ding, Lamaison, Liu, Wang, and Yang~\cite{DLLWY}]\label{conj:1.7}
For a linear $3$-graph $F$, $\pi_{\mathrm{co}}(F) = 0$ if and only if $F$ is layered.
\end{conjecture}

In this paper, we consider the codegree Tur\'an density for general non-layered $k$-graph. The main result is as below.
\begin{theorem}\label{thm:main}
Let $k\ge3$. If $F$ is a non-layered $k$-graph on $m$ vertices, then
\[
  \pico(F)\ge q_{k,m}^{-q_{k,m}}>0,
\]
where
\(
  q_{k,m}=((k-1)^{m+1}-1)/(k-2).
\)
\end{theorem}

Taking $k=3$ in Theorem~\ref{thm:main} together with Theorem~\ref{thm:known} confirms Conjectures \ref{conj:1.6} and \ref{conj:1.7}. 

\begin{corollary}\label{cor:characterization}
For every $3$-graph $F$,
$\pico(F)=0$ if and only if $F$ is layered and $\piunif(F)=0$;
For every linear $3$-graph $F$,
$\pico(F)=0$ if and only if $F$ is layered.
\end{corollary}

Since $\pi_{\ell} (F)\geq \pico (F)$ for $1\leq \ell \leq k-1$, by Theorem \ref{thm:main}, we have the following.
\begin{corollary}
    Let $1\leq \ell \leq k-1$ and $F$ be a $k$-graph. If $\pi_{\ell} (F)=0 $, then $F$ is layered.
\end{corollary}

The main idea for proving Theorem \ref{thm:main} is as follows. Firstly, if an $m$-vertex $k$-graph $F$ is layered, we can identify vertices at the same level and orient the edges from lower to higher levels. This layered condition can be translated into the consistency and acyclicity of a corresponding quotient digraph, which is much easier to verify. Secondly,
we assign to each vertex of $F$ a complete rooted $(k-1)$-ary tree of depth $m$, whose vertices are labeled such that vertices on the same path from the root receive distinct labels. If, for any edge of $F$, the $k$ corresponding labeled trees satisfy a specific relation, specifically, if one labeled tree records the truncated information of the other $k-1$ trees, then $F$ is layered. This property enables us to construct a $k$-graph $H_n$ with linear minimum codegree such that all of its subgraphs of order at most $m$ are layered, that is, $H_n$ contains no
non-layered subgraphs of order at most $m$.

\section{Quotient digraphs and labelled
\texorpdfstring{$(k-1)$}{(k-1)}-ary trees}

Throughout the rest of this paper, we always assume that $k\ge 3$ and $d=k-1$.  

A \emph{$d$-to-$1$ orientation} of a
$k$-graph $G$ chooses one vertex of every edge as its head.  If the head of
an edge $\{x_1,\ldots,x_d,z\}$ is $z$, we write the oriented edge as
\(
  x_1\cdots x_d\rightarrow z;
\)
the order of $x_1,\ldots,x_d$ is irrelevant.

Let $\mathcal Q$ be a partition of $V(G)$, and $Q(v)$ be the part
containing a given vertex $v$.  The loopless \emph{quotient digraph}
$D(\rbjt G,\mathcal Q)$ has vertex set $\mathcal Q$.  For distinct
$X,Z\in\mathcal Q$, it contains the arc $X\to Z$ if there is an oriented
edge $x_1\cdots x_d\to z$ such that $Q(z)=Z$ and
$X\in\{Q(x_1),\ldots,Q(x_d)\}$.

For a positive integer $n$, let $[n] = \{1,\ldots,n\}$. The following lemma translates the layered condition into a corresponding property of the quotient digraph $D(\rbjt G,\mathcal Q)$. Its first condition says that a head class determines one tail-class multiset, while the second condition allows the classes to be ordered from lower to higher levels.

\begin{lemma}\label{lem:quotient-characterization}
A $k$-graph $G$ is layered if and only if it has a $d$-to-$1$
orientation $\rbjt G$ and a partition $\mathcal Q$ of $V(G)$ satisfying the
following two conditions.
\begin{enumerate}
 \setlength{\itemsep}{1pt}
  \setlength{\parsep}{1pt}
 \item[(1)] If
  $
    x_1\cdots x_d\to z$ and
    $y_1\cdots y_d\to z'
  $
  are oriented edges with $Q(z)=Q(z')$, then
  $
    \{Q(x_1),\ldots,Q(x_d)\}
    =\{Q(y_1),\ldots,Q(y_d)\}
  $
  as multisets.
  \item[(2)] For every oriented edge $x_1\cdots x_d\to z$,
  $
    Q(z)\notin\{Q(x_1),\ldots,Q(x_d)\}.
 $
  And  $D(\rbjt G,\mathcal Q)$ is acyclic.
\end{enumerate}
\end{lemma}

\begin{proof}
Suppose that $f:V(G)\to\mathbb N$ is a layered function.  Let
$\mathcal Q$ consist of the nonempty classes $f^{-1}(i)$, so two vertices lie
in the same part exactly when they have the same $f$-label.  By (A1), every
edge has a unique vertex with the largest label; orient the edge towards that
vertex.

Consider two oriented edges
$
  x_1\cdots x_d\to z$ and $
  y_1\cdots y_d\to z'$
with $Q(z)=Q(z')$.  Then $f(z)=f(z')$, and these are the maximum labels of
the two edges.  By (A2),
$
  \{f(x_1),\ldots,f(x_d),f(z)\}
  =\{f(y_1),\ldots,f(y_d),f(z')\}$
as multisets. Deleting the common maximum label yields
$
  \{f(x_1),\ldots,f(x_d)\}
  =\{f(y_1),\ldots,f(y_d)\}.
$
Thus
$
  \{Q(x_1),\ldots,Q(x_d)\} \\
  =\{Q(y_1),\ldots,Q(y_d)\}.
$

For every oriented edge $x_1\cdots x_d\to z$, condition (A1) gives
$f(x_i)<f(z)$ for every $i\in[d]$.  Consequently,
$Q(z)\notin\{Q(x_1),\ldots,Q(x_d)\}$.  Every arc $X\to Z$ of
$D(\rbjt G,\mathcal Q)$ comes from an oriented edge whose tail has a vertex in
$X$ and whose head lies in $Z$.  The common $f$-label on $X$ is therefore
smaller than the common $f$-label on $Z$.  Thus the labels strictly increase
along every directed path, so the quotient digraph is acyclic.

Conversely, suppose that $\rbjt G$ and $\mathcal Q$ satisfy the two stated
conditions.  Since $D(\rbjt G,\mathcal Q)$ is acyclic, it has a topological
ordering~\cite{Bang-Jensen}; that is, there is a bijection
$
  \lambda:\mathcal Q\to[|\mathcal Q|]
$
such that $\lambda(X)<\lambda(Z)$ whenever $X\to Z$ is an arc of
$D(\rbjt G,\mathcal Q)$.  Define
$
  f(v)=\lambda(Q(v))
$
for every $v\in V(G)$.

Let $x_1\cdots x_d\to z$ be an oriented edge. Since $Q(x_i)\ne Q(z)$ for every $i\in[d]$,
$Q(x_i)\to Q(z)$ is an arc of $D(\rbjt G,\mathcal Q)$. Therefore,
$f(x_i)<f(z)$ for every $i\in[d]$, which implies that (A1) holds.

Now consider two edges with the same maximum $f$-label, and write their
oriented forms as
$x_1\cdots x_d\to z$ and $y_1\cdots y_d\to z'.$
Then $\lambda(Q(z))=\lambda(Q(z'))$.  Since $\lambda$ is injective,
$Q(z)=Q(z')$.  The first condition gives
$
  \{Q(x_1),\ldots,Q(x_d)\}
  =\{Q(y_1),\ldots,Q(y_d)\}
$
as multisets.  Applying $\lambda$ to the tail classes and using
$f(z)=f(z')$, we obtain equal label multisets on the two edges.  Thus
(A2) holds.  Hence $f$ is layered.
\end{proof}

Fix $m\ge1$ and let
$q_{k,m}=1+d+\cdots+d^m
  =(d^{m+1}-1)/(d-1).$
For $0\le r\le m$, let $T_r$
be the complete rooted $d$-ary tree of depth $r$ with 
$ V(T_r)=\cup_{j=0}^{r}[d]^j,$
where $[d]^0=\{\varnothing\}$ and $[d]^j=\{a_1\cdots a_j \mid a_1,\ldots,a_j\in[d]\} $ is the set of all $d$-ary words of length $j$ for $j\ge 1$. For given two words $w$ and $i$, let $wi$ be the word obtained by placing $i$ after $w$. The root of $T_r$ is the empty word $\varnothing$, and the $d$ children of a word $w$ of length less than $r$ are $w1,\ldots,wd$. 
Thus the vertices at depth $j$ are exactly the words of length $j$, and
\(
  |T_r|=1+d+\cdots+d^r.
\)

An \emph{admissible labeling} of $T_r$ is a map $A:T_r\to[q_{k,m}]$
such that $A(u)\ne A(v)$ whenever $u$ is a proper prefix of $v$.
Equivalently, no label is repeated on a path from the root to a vertex.
Labels may still be repeated in different branches.  Let $\mathcal C_r$ be
the finite collection of all admissible labelings of $T_r$.  Thus each
element of $\mathcal C_r$ is an entire labelled tree.

For $A\in\mathcal C_r$ and $0\le s\le r$, write $A|_s$ for the restriction
of $A$ to $T_s$.  In other words, $A|_s$ is obtained by deleting the vertices
at depths $s+1,\ldots,r$.  In particular, $A|_s\in\mathcal C_s$.

Suppose that $r\ge1$ and $C\in\mathcal C_r$.  The $d$ subtrees rooted at the
children $1,\ldots,d$ of the root all have depth $r-1$.  We identify each of
these subtrees with $T_{r-1}$.  For $i\in[d]$, define
\[
  C^i:T_{r-1}\to[q_{k,m}],
  \quad C^i(w)=C(iw),
\]
Thus $C^i(\varnothing)=C(i)$, $C^i(1)=C(i1)$. Under the
identification $w\leftrightarrow iw$, the map $C^i$ is exactly the labeling
of the subtree rooted at $i$.  Since placing the same first letter before two
words preserves the ancestor relation, the admissibility of $C$ implies that
$C^i$ is admissible.  Hence $C^1,\ldots,C^d\in\mathcal C_{r-1}$; we call them
the main branches of $C$.

For $1\le r\le m$ and $A_1,\ldots,A_d,~C\in\mathcal C_r$, we write
\begin{equation}\label{eq:tree-relation}
  A_1\cdots A_d\rightarrow_r C
  \quad\Longleftrightarrow\quad
  \{C^1,\ldots,C^d\}
  =\{A_1|_{r-1},\ldots,A_d|_{r-1}\}
  \quad\text{as multisets}.
\end{equation}
Thus the main branches of $C$ are precisely the restrictions of
$A_1,\ldots,A_d$ to their first $r$ levels.  Equality as multisets means that
the branches may be permuted, while repeated branches are retained.  In
particular, the order of $A_1,\ldots,A_d$ is irrelevant, and repetitions are
allowed.  We may therefore regard $C$ as recording the truncated information
of $A_1,\ldots,A_d$ in its $d$ main branches.

\begin{lemma}\label{lem:tree-basic}
The relation in~\eqref{eq:tree-relation} has the following properties.
\begin{enumerate}
 \setlength{\itemsep}{1pt}
  \setlength{\parsep}{1pt}
  \setlength{\parskip}{1pt}
 
  \item[(1)] If $2\le r\le m$ and $A_1\cdots A_d\to_r C$, then $(A_1|_{r-1})\cdots(A_d|_{r-1})
    \to_{r-1} C|_{r-1}.$
  \item[(2)] If $1\le r\le m$ and $A_1\cdots A_d\to_r C$, then $C\notin\{A_1,\ldots,A_d\}.$
  \item[(3)] For every $A_1,\ldots,A_d\in\mathcal C_m$, there is
  $C\in\mathcal C_m$ such that
  $A_1\cdots A_d\to_m C.$
\end{enumerate}
\end{lemma}

\begin{proof}
For the first statement, the main branches of $C|_{r-1}$ are
$(C^1)|_{r-2},\ldots,(C^d)|_{r-2}$; indeed,
$(C|_{r-1})^i=(C^i)|_{r-2}$ for every $i\in[d]$.  Since
$A_1\cdots A_d\to_r C$, these trees are
$A_1|_{r-2},\ldots,A_d|_{r-2}$ in some order.  This is precisely the
relation
$(A_1|_{r-1})\cdots(A_d|_{r-1})
  \to_{r-1} C|_{r-1}$
required in the first statement.

For the second statement, suppose that $C=A_j$ for some $j\in[d]$.  
By the definition of $A_1\cdots A_d\to_r C$, one of the main branches of $C$, say $C^i$, is equal to $A_j|_{r-1}=C|_{r-1}$.
Evaluating these two labelled trees at their roots gives $C(i)=C^i(\varnothing)=C|_{r-1}(\varnothing)=C(\varnothing)$.
But $\varnothing$ is a proper prefix of the one-letter word $i$, contradicting the admissibility of $C$.

For the third statement, each tree $A_j|_{m-1}$ has
$1+d+\cdots+d^{m-1}$ vertices.  Hence the $d$ trees together use at most $d(1+d+\cdots+d^{m-1})=d+d^2+\cdots+d^m=q_{k,m}-1$
different labels. Choose
$\gamma\in[q_{k,m}]$ that occurs in none of them.  Label a new root by $\gamma$
and attach $A_1|_{m-1},\ldots,A_d|_{m-1}$ below it as the $d$ main branches.
Every root-to-vertex path is admissibly labelled.  The resulting tree $C$ belongs to $\mathcal C_m$ and
satisfies $A_1\cdots A_d\to_m C$.
\end{proof}

\begin{lemma}\label{lem:merging}
Let $\rbjt G$ be a $d$-to-$1$ orientation of a $k$-graph $G$ with at most $m$
vertices.  Suppose that there is a map $\chi:V(G)\to\mathcal C_m$ such that $\chi(x_1)\cdots\chi(x_d)\to_m\chi(z)$
for every oriented edge $x_1\cdots x_d\to z$.  Then $V(G)$ has a partition
$\mathcal Q$ satisfying the two conditions in
Lemma~\ref{lem:quotient-characterization}.
\end{lemma}

\begin{proof}
We now construct the required partition.  Throughout the construction, we
maintain a partition $\mathcal Q$ of $V(G)$, an integer $1\le s\le m$, and a
map $\psi:\mathcal Q\to\mathcal C_s$ such that
\begin{equation}\label{eq:merging-invariant}
  |\mathcal Q|\le s
  \quad\text{and}\quad
  \psi(Q(x_1))\cdots\psi(Q(x_d))
  \to_s\psi(Q(z))
\end{equation}
for every oriented edge $x_1\cdots x_d\to z$.
Initially, let $\mathcal Q$ be the partition into singleton parts, let $s=m$,
and define $\psi(\{v\})=\chi(v)$ for every $v\in V(G)$.  Since
$|V(G)|\le m$, the conditions in~\eqref{eq:merging-invariant} hold.

Suppose that the first condition of Lemma~\ref{lem:quotient-characterization} fails for the current partition.
Then there are oriented edges $x_1\cdots x_d\to z$ and $y_1\cdots y_d\to z'$ such that $Q(z)=Q(z')$ but $\{Q(x_1),\ldots,Q(x_d)\}\ne\{Q(y_1),\ldots,Q(y_d)\}$ as multisets. Since $|\mathcal Q|\ge2$, we have $s\ge 2$. 
The corresponding relations in \eqref{eq:merging-invariant} have the same head tree. Hence, after interchanging $y_1,\ldots,y_d$ if necessary, we have
\begin{equation}\label{eq:matching-branches}
  \psi(Q(x_i))|_{s-1}
  =\psi(Q(y_i))|_{s-1}
  \quad \text{for} \quad i\in[d].
\end{equation}
Because the two tail-part multisets are different, $Q(x_i)\ne Q(y_i)$ for at least one $i\in[d]$.  

Let $\sim$ be the equivalence relation on the parts of $\mathcal Q$ generated by
$Q(x_i)\sim Q(y_i)$ for $i\in[d]$.
Let $\mathcal Q'$ be the partition obtained by taking the unions of the
$\sim$-equivalence classes.  Equivalently, $\mathcal Q'$ is the finest
coarsening of $\mathcal Q$ in which all these identifications hold.

By~\eqref{eq:matching-branches}, every generating identification joins two old parts whose tree labels have the same restriction to $T_{s-1}$.  
By transitivity, all old parts contained in one part of $\mathcal Q'$ have the same restriction to $T_{s-1}$.  
We may therefore define $\psi'(X')=\psi(X)|_{s-1}$, where $X$ is any old part contained in $X'$.  This definition is well defined.

For every oriented edge $a_1\cdots a_d\to c$, the invariant
\eqref{eq:merging-invariant} and Lemma~\ref{lem:tree-basic}(1) give
\[
  \psi'(Q'(a_1))\cdots\psi'(Q'(a_d))
  \to_{s-1}\psi'(Q'(c)).
\]
Moreover, at least one pair of distinct old parts has been identified, so $|\mathcal Q'|\le|\mathcal Q|-1\le s-1$.
Thus the invariant remains valid after replacing $(\mathcal Q,s,\psi)$ by $(\mathcal Q',s-1,\psi')$.

Each merging step strictly decreases the number of parts.  Hence the process
terminates after finitely many steps.
For the terminal partition, the first condition of
Lemma~\ref{lem:quotient-characterization} holds.

Let $x_1\cdots x_d\to z$ be any oriented edge.  By \eqref{eq:merging-invariant} and Lemma~\ref{lem:tree-basic}(2), $\psi(Q(z))\notin\{\psi(Q(x_1)),\ldots,\psi(Q(x_d))\}$.
Therefore $Q(z)\notin\{Q(x_1),\ldots,Q(x_d)\}$.

	\vskip 0.2em

It remains to show that $D(\rbjt G,\mathcal Q)$ is acyclic.  Suppose otherwise, and choose a directed cycle of minimum length: $X_0\to X_1\to\cdots\to X_\ell=X_0$.
Then $2\le\ell\le|\mathcal Q|\le s$.

For each $0\le i<\ell$, the arc $X_i\to X_{i+1}$ comes from an oriented edge whose head lies in $X_{i+1}$ and one of whose tail vertices lies in $X_i$.
It follows from~\eqref{eq:merging-invariant} that there is $\varepsilon_i\in[d]$ such that
\begin{equation}\label{eq:cycle-branch}
  \psi(X_{i+1})^{\varepsilon_i}
  =\psi(X_i)|_{s-1}.
\end{equation}

Let $w_0=\varnothing$ and $\alpha=\psi(X_0)(\varnothing)$.  We prove inductively that, for every $0\le i\le\ell$, there is a word $w_i$ of length $i$ over $[d]$ such that $\psi(X_i)(w_i)=\alpha$.
This is true for $i=0$.  Suppose it holds for some $i<\ell$.  Since $i\le\ell-1\le s-1$, the word $w_i$ belongs to $T_{s-1}$.  Define $w_{i+1}=\varepsilon_iw_i$.
Using~\eqref{eq:cycle-branch}, we obtain
\[
  \psi(X_{i+1})(w_{i+1})
  =\psi(X_{i+1})^{\varepsilon_i}(w_i)
  =\psi(X_i)(w_i)
  =\alpha.
\]
This proves the assertion.

Finally, since $X_\ell=X_0$, the label $\alpha$ occurs in $\psi(X_0)$ both
at the root $\varnothing$ and at the vertex $w_\ell$ of depth $\ell\ge1$.
Since $\varnothing$ is a proper prefix of $w_\ell$, this contradicts the
admissibility of $\psi(X_0)$.  Therefore, $D(\rbjt G,\mathcal Q)$ is acyclic.

Thus $\mathcal Q$ satisfies both conditions in
Lemma~\ref{lem:quotient-characterization}, completing the proof.
\end{proof}

\section{Proof of Theorem~\ref{thm:main}}\label{sec3}

Let $\mathcal P_{k,m}$ be the collection of all $k$-element multisets  $\{A_1,\ldots,A_k\}$ with $A_1,\ldots,A_k\in\mathcal C_m$ for which at
least one of
\[
  A_1\cdots A_{k-1}\to_m A_k,\quad A_1\cdots A_{k-2}A_{k}\to_m A_{k-1},\quad \ldots  \quad ,\quad A_2\cdots A_k\to_m A_1
\]
holds.  
The order of $A_1,\ldots,A_k$ is irrelevant, and they need not be distinct.

A $\mathcal P_{k,m}$-coloring of a $k$-graph $G$ is a map $\chi:V(G)\to\mathcal C_m$ such that for every edge $x_1x_2\cdots x_k\in E(G)$, $\{\chi(x_1),\chi(x_2),\ldots,\chi(x_k)\}\in\mathcal P_{k,m}$ .

\begin{lemma}\label{thm:finite-coloring}
Let $m\ge k\ge 3$.
Every $k$-graph on at most $m$ vertices that admits a
  $\mathcal P_{k,m}$-coloring is layered.
\end{lemma}

\begin{proof}
Let $G$ be a $k$-graph with $|V(G)|\le m$, and let
$\chi:V(G)\to\mathcal C_m$ be a $\mathcal P_{k,m}$-coloring.  Consider an
edge $e\in E(G)$.  Since the multiset of colors on $e$ belongs to
$\mathcal P_{k,m}$, we can choose one vertex $z\in e$ as a head and list the
remaining vertices as $x_1,\ldots,x_d$ so that $\chi(x_1)\cdots\chi(x_d)\to_m\chi(z)$.
Orient the edge as $x_1\cdots x_d\to z$.  If more than one choice is
possible, choose any one of them.

In this way, every edge of $G$ receives a $d$-to-$1$ orientation, and every
oriented edge $x_1\cdots x_d\to z$ satisfies
\[
  \chi(x_1)\cdots\chi(x_d)\to_m\chi(z).
\]
Since $|V(G)|\le m$, Lemma~\ref{lem:merging} provides a partition
$\mathcal Q$ of $V(G)$ satisfying the two conditions in
Lemma~\ref{lem:quotient-characterization}, which implies that $G$ is layered.
\end{proof}

We now prove Theorem~\ref{thm:main} by constructing a $k$-graph $H_n$ with linear minimum codegree such that all of its subgraphs of order at most $m$ are layered.

\begin{proof}[\bfseries{Proof of Theorem~\ref{thm:main}}]
Let $F$ be a non-layered $k$-graph on $m$ vertices.  By
Lemma~\ref{thm:finite-coloring}, $F$ has no
$\mathcal P_{k,m}$-coloring.

We define an $n$ vertex $k$-graph $H_n$ as follows.
Partition the vertex set of $H_n$ into classes
$V_A$, indexed by $A\in\mathcal C_m$, whose sizes differ by at most one.
Thus $|V_A|\geq\lfloor n/|\mathcal C_m|\rfloor$ for every $A\in\mathcal C_m$. Any $k$
distinct vertices form an edge of $H_n$ if and only if the multiset of the labels
of their classes belongs to $\mathcal P_{k,m}$.

Let $S=\{x_1,\ldots,x_d\}$ be an arbitrary $(k-1)$-subset of vertices, and suppose that $x_i\in V_{A_i}$ for $i\in[d]$.  The colors $A_1,\ldots,A_d$ may be repeated.  By Lemma~\ref{lem:tree-basic}(3), there is some $C\in\mathcal C_m$ such that $A_1\cdots A_d\to_m C$.
By the definition of $\mathcal P_{k,m}$, this gives $\{A_1,\ldots,A_d,C\}\in\mathcal P_{k,m}$. Moreover, Lemma~\ref{lem:tree-basic}(2) gives
$C\notin\{A_1,\ldots,A_d\}$.  Hence $V_C\cap S=\varnothing$, and every
vertex of $V_C$ is a common neighbor of $S$.  We therefore obtain
\[
\delta_{\mathrm{co}}(H_n)\ge\left\lfloor\frac{n}{|\mathcal C_m|}\right\rfloor.
\]

We next show that $H_n$ is $F$-free.  Suppose, for the sake of contradiction,
that $H_n$ contains a copy of $F$.  Assign to each vertex of this copy the
label of the class containing it.  For every edge of the copy, the multiset
of its $k$ class labels belongs to $\mathcal P_{k,m}$ by the definition of
$H_n$.  This assignment is therefore a $\mathcal P_{k,m}$-coloring of $F$,
contrary to the choice of $F$.  Hence $H_n$ is $F$-free. Therefore,
\[
  \pico(F)\ge\frac{1}{|\mathcal C_m|}.
\]

The tree $T_m$ has $1+d+\cdots+d^m=q_{k,m}$ vertices, and each of its vertices has a label in $[q_{k,m}]$.  Hence the total number of labelings of $T_m$ is $q_{k,m}^{q_{k,m}}$.  Since $\mathcal C_m$ consists only of the admissible labelings, it follows that $|\mathcal C_m|\le q_{k,m}^{q_{k,m}}$.
Consequently, $$\pico(F)\ge q_{k,m}^{-q_{k,m}}>0.$$ This completes the proof of Theorem~\ref{thm:main}.
\end{proof}

\begin{proof}[\bfseries{Proof of Corollary~\ref{cor:characterization}}]
Suppose first that $F$ is a $3$-graph with $\pico(F)=0$.  The  case $k=3$ of Theorem~\ref{thm:main} shows that $F$ is layered. And then $\piunif(F)=0$ by Theorem~\ref{thm:known}(2).  
Conversely, if $F$ is layered and $\piunif(F)=0$, Theorem~\ref{thm:known}(2) implies that $\pico(F)=0$.  This proves the first equivalence.

By the characterization of Reiher, R\"odl and Schacht~\cite{RRS}, every linear $3$-graph has zero uniform Tur\'an density.  The second equivalence
follows from the first.  
\end{proof}

\section{Concluding remarks}\label{sec4}

For a layered function $f$ of a $k$-graph $F$, 
let $|f(V(F))|$ be the number of distinct elements in $\{ f(v) \mid v\in V(F) \}$. We call $f$ \emph{minimum} if $|f(V(F))|$ is minimum among
all layered functions of $F$. 
The minimal layered function of $F$ has the following property.

\begin{proposition}\label{layered}
The minimal layered function of $F$ satisfies {\em (A3)}.
\end{proposition}
\begin{proof}
Let $f$ be a minimum layered function of $F$. Suppose that $f$ does not satisfy (A3). Then there exist $e,e'\in E(F)$ such that
\[
 |f(e)\cap f(e')|\ge k-1
 \qquad\text{and}\qquad
 f(e)\ne f(e').
\]
By \textnormal{(A2)},
$\max f(e)\ne\max f(e')$. After interchanging $e$ and $e'$ if necessary,
let $p=\max f(e)>t:=\max f(e')$. Then $p\notin f(e')$, and
\textnormal{(A1)} says that $p$ occurs exactly once in $f(e)$. Hence, for
some multiset $C$ and some $q\le t$, we may write
\[
 f(e)=C\cup\{p\},\qquad
 f(e')=C\cup\{q\},\qquad
 p>t\ge q.
\]
Define $g\colon V(F)\to\mathbb N$ by
\[
 g(v)=
 \begin{cases}
  q,&f(v)=p,\\
  f(v),&f(v)\ne p.
 \end{cases}
\]
Thus $g(e)=g(e')=f(e')$.

Let $a\in E(F)$. If $\max f(a)>p$, its unique maximum is unchanged; if
$\max f(a)<p$, then $g(a)=f(a)$; and if $\max f(a)=p$, then
\textnormal{(A2)} gives $f(a)=f(e)$, so $g(a)=f(e')$, which has a unique
maximum. Hence $g$ satisfies \textnormal{(A1)}, and
\begin{equation}\label{eqmax}
 \max g(a)=
 \begin{cases}
  t,&\max f(a)=p,\\
  \max f(a),&\max f(a)\ne p.
 \end{cases}
\end{equation}

Now suppose that $\max g(a)=\max g(b)=\ell$. If $\ell\ne t$, then
(\ref{eqmax}) yields $\max f(a)=\max f(b)=\ell$, so \textnormal{(A2)} for $f$
implies $g(a)=g(b)$. If $\ell=t$, then each of $f(a)$ and $f(b)$ equals
either $f(e)$ or $f(e')$; under the replacement $p\mapsto q$, both become
$f(e')$. Thus $g(a)=g(b)$, and $g$ also satisfies \textnormal{(A2)}.

Therefore, $g$ is layered. However, $q$ already occurs in $f(e')$ while
every occurrence of $p$ has been removed. Hence
$|g(V(F))|=|f(V(F))|-1,$
contradicting the minimality of $f$. Thus $f$ satisfies
\textnormal{(A3)}.
\end{proof}

\section*{Acknowledgement}
\noindent This research is supported by the National Key R\&D Program of China under grant number 2024YFA1013900, the NSFC under grant number 12471327 and the China Postdoctoral Science Foundation under grant number 2026M793375.

\section*{Declaration}
\noindent\textbf{Conflict of interest.}
The authors declare that they have no known competing financial interests or personal relationships that could have appeared to influence the work reported in this paper.

\section*{Data availability}
No data were used for the research described in this paper.



\end{document}